\documentclass[12pt,a4paper]{amsart}
\usepackage[utf8]{inputenc}
\usepackage{fullpage}
\usepackage{amsfonts}
\usepackage{amssymb}
\usepackage{enumitem}
\usepackage{graphicx}
\usepackage{url}
\usepackage{hyperref}
\usepackage[capitalise,noabbrev]{cleveref}
\usepackage{amsmath}
\usepackage{amsthm}
\usepackage{color}
\usepackage{todonotes}

\usepackage[backend=bibtex,block=space,sortcites=true,maxbibnames=99,style=alphabetic]{biblatex}
\usepackage{doi}
\renewbibmacro{in:}{}
\DeclareFieldFormat[article]{volume}{\mkbibbold{#1}}
\DeclareFieldFormat[article]{number}{(#1)}
\DeclareFieldFormat{doi}{\href{https://dx.doi.org/#1}{\textsc{doi}}}
\AtEveryBibitem{\clearfield{issn}\clearfield{isbn}\clearlist{language}}
\renewbibmacro*{volume+number+eid}{%
  \printfield{volume}%
  \printfield{number}%
  \setunit{\bibeidpunct}%
  \printfield{eid}}
\def\abx@missing@entry#1{\abx@missing{#1??}}

\newtheorem{theorem}{Theorem}[section]
\newtheorem{lemma}[theorem]{Lemma}
\newtheorem{prop}[theorem]{Proposition}
\newtheorem{cor}[theorem]{Corollary}  

\newtheorem{conj}[theorem]{Conjecture}
\newtheorem{defn}[theorem]{Definition}

\theoremstyle{definition}
\newtheorem{rem}[theorem]{Remark}

\crefname{theorem}{Theorem}{Theorems}
\crefname{prop}{Proposition}{Propositions}
\crefname{lemma}{Lemma}{Lemmas}
\crefname{cor}{Corollary}{Corollaries}
\crefname{rem}{Remark}{Remarks}
\crefname{hyp}{Hypothesis}{Hypotheses}
\crefname{conj}{Conjecture}{Conjectures}
\crefname{defn}{Definition}{Definitions}
\crefname{constr}{Construction}{Constructions}
\crefname{enumi}{}{} 

\newcommand{\alf}{\operatorname{asc}} 
\newcommand{\dyck}{\mathcal{D}} 
\newcommand{\mdyck}{\mathcal{D}^{(m)}} 
\newcommand{\leqgreedy}{\leq_g} 
\newcommand{\covgreedy}{\lessdot_g} 
\newcommand{\tree}{\mathcal{T}} 
\newcommand{\mtree}{\mathcal{T}^{(m)}} 
\newcommand{\mrtree}{\mathcal{T}^{(m,r)}} 
\newcommand{\minterval}{\mathcal{I}^{(m)}} 
\newcommand{\mrinterval}[1][r]{\mathcal{I}^{(m,#1)}} 
\newcommand{\interval}{\mathcal{I}} 
\newcommand{\firstasc}{\operatorname{fasc}} 
\newcommand{\mconst}{\mathcal{C}^{(m)}} 
\newcommand{\map}{\mathcal{M}} 
\newcommand{\outdeg}{\operatorname{outdeg}} 
\newcommand{\portion}{\operatorname{Pt}} 
\newcommand{\straight}{r} 
\newcommand{\midtree}{\operatorname{mid}} 
\newcommand{\addleaf}{\Delta} 
\newcommand{\intweight}{\operatorname{wt}} 
\newcommand{\intproduct}{\oplus} 
\newcommand{\mintweight}{\operatorname{wt}_m} 
\newcommand{\mrintweight}{\operatorname{wt}_{m, r}} 

\newcommand{\tdef}[1]{\textcolor{blue}{\emph{#1}}}
\newcommand{\insertfig}[2]{\includegraphics[page=#1, width=#2\textwidth]{fig-ipe.pdf}}

\title{Decomposition of Greedy Tamari Intervals and Bipartite Planar Maps}

\author{Philippe Biane}
\address{Univ Gustave Eiffel, CNRS, LIGM, F-77454 Marne-la-Vall\'e, France}
\email{philippe.biane@univ-eiffel.fr}
\thanks{Supported by the project CARPLO (ANR-20-CE40-0007).}

\author{Wenjie Fang}
\address{Univ Gustave Eiffel, CNRS, LIGM, F-77454 Marne-la-Vall\'ee, France}
\email{wenjie.fang@univ-eiffel.fr}
\thanks{Supported by ANR--FWF project PAGCAP (ANR-21-CE48-0020) and CartesEtPlus (ANR-23-CE48-0018).}

\begin{document}

\begin{abstract}
  The greedy Tamari poset, inspired by the well-studied Tamari lattice, was recently defined by Dermenjian in the more general setting of greedy $\nu$-Tamari posets. Bousquet-Mélou and Chapoton counted intervals of the greedy $m$-Tamari poset in 2024 by solving a functional equation, and found that they are equi-enumerous to planar $(m+1)$-constellations. In this work, we give a combinatorial proof of this fact for the case $m = 1$, which also gives the refined enumeration conjectured by Bousquet-Mélou and Chapoton. This is done by establishing a recursive decomposition of greedy Tamari intervals isomorphic to that of bipartite planar maps. We also propose a more general and refined conjecture for the case of general $m$.
\end{abstract}

\maketitle

\section{Introduction}

The Tamari lattice is a much studied object in combinatorics and geometry, whose definition goes back to the work of Dov Tamari in 1962 \cite{tamari}, see \cite{tam-book} for information on this subject. Later, generalizations of the Tamari lattice, such as the $m$-Tamari lattice \cite{m-tamari-def} and the $\nu$-Tamari lattice \cite{nu-tamari}, were proposed with motivation from other domains of combinatorics, such as the study of diagonal coinvariant spaces \cite{m-tamari-def}. It has been observed, since the seminal work of Chapoton \cite{chapoton-counting}, that counting intervals in the Tamari lattice and its variants often yields formulas that also count objects related to planar maps, which are drawings of graphs on the plane. Such results are often first obtained by solving functional equations, as in \cite{chapoton-counting}, then reproved bijectively, as in \cite{bernardi-bonichon,fang-tamari,tamari-blossoming}. However, there is the exception of the intervals of  $m$-Tamari lattices. Their counting formula obtained in \cite{m-tamari} looks close to those for the enumeration of planar maps, but no combinatorial proof is known, nor any natural family of planar maps with the same counting formula. It is thus a surprising news that Bousquet-M\'elou and Chapoton found and proved in \cite{greedy-tamari-interval} that the number of intervals in the greedy $m$-Tamari poset, whose definition is inspired by that of the $m$-Tamari lattice, is equal to the number of certain planar maps.

The greedy $m$-Tamari poset was considered recently by Dermenjian in \cite{dermenjian} as a special case of the greedy $\nu$-Tamari poset, defined on the same ground set as $\nu$-Tamari lattices, but with cover relations produced by a ``greedy'' version of the rules for $\nu$-Tamari lattices. When defined on $m$-ballot paths of the same size, we obtain the greedy $m$-Tamari poset, which was shown in \cite[Theorem~34]{dermenjian} to be isomorphic to the subposet of elements with maximal in-degree in the $(m+1)$-Tamari lattice. Soon afterwards, Bousquet-M\'elou and Chapoton enumerated in \cite{greedy-tamari-interval} the number of intervals in the greedy $m$-Tamari poset by solving a functional equation based on a recursive decomposition of such intervals. They found that the counting formula is the same as that of planar $(m+1)$-constellations, a family of planar maps first enumerated in \cite{constellation} using bijections. In \cite[Conjecture~6.1]{greedy-tamari-interval}, they also conjectured that a certain refined enumeration of greedy $m$-Tamari intervals is the same as that of planar $(m+1)$-constellations refined by the profile of certain face degrees.

Our main contribution in this article is proving \cite[Conjecture~6.1]{greedy-tamari-interval} combinatorially for $m = 1$ by exhibiting a recursive decomposition of greedy Tamari intervals that is isomorphic to the classical recursive decomposition of planar bipartite maps from Tutte \cite{tutte-slicing}. Such an isomorphism gives naturally a recursively-defined bijection between greedy Tamari intervals and planar bipartite maps. Furthermore, we show that the statistics used in the refined counting on each side are transferred according to the conjecture. The greedy Tamari poset was originally defined on lattice paths, but in our proof we realize it on a certain family of plane trees, which provides a better understanding of the intervals we study. Unfortunately, we did not succeed in adapting our approach to the general $m$ case. However, we will propose a generalized and refined version of \cite[Conjecture~6.1]{greedy-tamari-interval} related to the recursive decomposition of constellations in \cite[Section~4.1]{fang-thesis}.

This article is organized as follows. In \cref{sec:defn}, we give the definitions of the greedy Tamari poset and its intervals, using both Dyck paths and plane trees. We also define bipartite planar maps and their recursive decomposition from Tutte. We then present in \cref{sec:decomp} the recursive decomposition of greedy Tamari intervals, which leads to the recursively-defined bijection between these intervals and bipartite maps. Finally, in \cref{sec:m-decomp} we consider the $m$-versions of the previous objects (namely $m$-greedy Tamari posets and $(m+1)$-constellations) and we state a generalized and refined conjecture for a bijection between these objects.

\section{Greedy Tamari intervals and bipartite maps} \label{sec:defn}

\subsection{The greedy Tamari poset on Dyck paths} \label{sec:greedy-dyck}

We start by defining the greedy Tamari poset over Dyck paths following 
\cite{dermenjian,greedy-tamari-interval}. In the following, we denote by $|w|_a$ the number of letters $a$ in a word $w$, and by $\cdot$ the concatenation of words.

A \tdef{Dyck path} of size $n$ is a word $w$ on the alphabet $\{0,1\}$ with $|w|_0 = |w|_1 = n$, such that for each prefix $w'$ of $w$, we have $|w'|_0 \leq |w'|_1$. We denote by $\dyck_n$ the set of Dyck paths with $2n$ letters. Such $w \in \dyck_n$ is usually represented as a lattice path on $\mathbb{Z}^2$ starting from $(0, 0)$, with \tdef{up steps} $(1, -1)$ (resp. \tdef{down steps} $(1, 1)$) standing for the letter $1$ (resp. $0$), while staying in the upper half plane $y \geq 0$ and ending at $(2n,0)$. An \tdef{ascent} (resp. \tdef{descent}) of a Dyck path $w$ is a maximal segment of consecutive $1$'s (resp. $0$'s) in $w$, while a \tdef{valley} (resp. \tdef{peak}) is an occurrence of the factor $01$ (resp. $10$). 

\begin{defn} \label{defn:greedy-poset}
  For $w \in \dyck_n$ with a marked valley, we write $w = A \cdot 01 \cdot A'$, where $01$ is the valley. Let $w^\circ$ be the longest prefix of $1 \cdot A'$ that is also a Dyck path, then $w$ has a factorization $w = A \cdot 0 \cdot w^\circ \cdot A''$. Let $w' = A \cdot w^\circ \cdot 0 \cdot A''$, this is clearly a Dyck path and we  declare that $w \covgreedy w'$. The relation $\covgreedy$ thus obtained by considering  all $w \in \dyck_n$ with all possible marked valleys, is the cover relation of an order relation $(\dyck_n, \leqgreedy)$ the transitive closure of $\covgreedy$.  The \tdef{greedy Tamari poset} is $(\dyck_n, \leqgreedy)$. The relation $\leqgreedy$ is also called the \tdef{greedy Tamari order}. See \Cref{fig:greedy-tamari} for an example of the cover relation $\covgreedy$ and the greedy Tamari poset of size $4$.
\end{defn}

\begin{figure}
  \centering
  \insertfig{9}{1}
  \caption{Example of the cover relation $\covgreedy$, and the greedy Tamari poset $(\dyck_4, \leqgreedy)$ of size $4$}
  \label{fig:greedy-tamari}
\end{figure}

There is a unique maximal element in $(\dyck_n, \leqgreedy)$, namely the path $1^n 0^{n}$, as it is the only element without valley, thus never covered. However, unlike the classical Tamari lattice, in general $(\dyck_n, \leqgreedy)$ does not have a unique minimal element, hence it is not a lattice.

Note that, if in the previous description we take $w^\circ$ to be the shortest prefix, then we get the cover relation for the classical Tamari order.
This implies that the greedy Tamari order is coarser that the classical order. 

\subsection{Plane trees and Dyck paths} \label{sec:plane-tree}

We denote by $\tree_n$ the set of rooted plane trees with $n$ edges, which are always drawn with the root at the bottom. There is a well-known bijection between Dyck paths in $\dyck_n$ and plane trees in $\tree_n$.

\begin{defn} \label{defn:contour-path}
  Given a plane tree $T$, its \tdef{contour path}, denoted by $w_T$, is the Dyck path obtained from the left-to-right contour walk of $T$, where the first traversal (resp. last traversal) of an edge is transcribed into a letter $1$ (resp. $0$). It is clear that the leaves of $T$ correspond bijectively to the peaks of $w_T$. See \Cref{fig:tree-path-stat} for an example.
\end{defn}

\begin{figure}
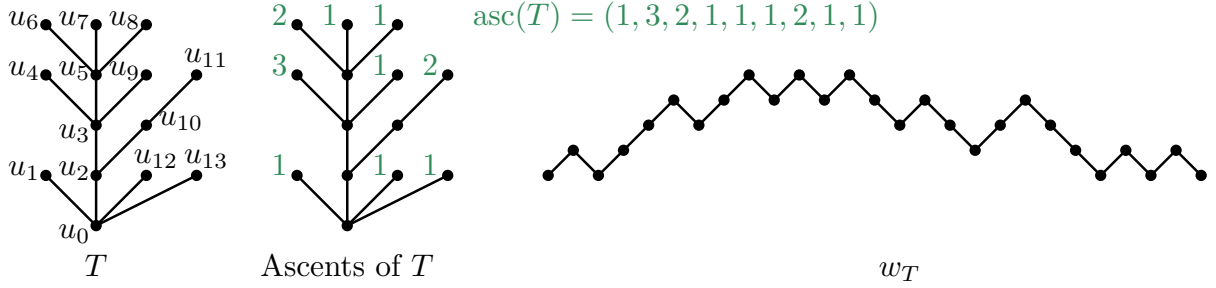

  \centering
  \insertfig{10}{1}
  \caption{Example of a plane tree $T$, its ascent profile $\alf(T)$ and its contour walk $w_T$}
  \label{fig:tree-path-stat}
\end{figure}
 
In the following, for a given plane tree $T \in \tree_n$, we denote its nodes by $u_0, u_1, \ldots, u_n$ in the order of first visit in the left-to-right depth-first search order, that we call the \tdef{pre-order} hereinafter. See \Cref{fig:tree-path-stat} for an illustration. In this context, we also say that $u_i$ is the $i$-th node of $T$ in the pre-order. In particular, the $0$-th node is always the root. For $i > 0$, we denote by $e_i$ the edge linking $u_i$ to its parent. It is clear that, in the left-to-right contour walk of $T$, edges of $T$ are visited for the first time in the order $e_1, \ldots, e_n$.
 
For $u_i$ a leaf of $T$, let $j \geq 0$ be the smallest index such that, for all $\ell$ with $j \leq \ell < i$, the node $u_{\ell}$ is the parent of $u_{\ell+1}$. We call the sequence of edges $(\{u_{\ell}, u_{\ell+1}\})_{j \leq \ell < i}$ the \tdef{ascent} of $u_i$, which is of length $i - j$. It is clear from \Cref{defn:contour-path} that each ascent of $T$ corresponds to an ascent of $w_T$ of the same length. We denote by $\firstasc(T)$ the length of the only ascent of $T$ containing the root $u_0$, which corresponds to the first ascent of $w_T$. The ascents of $T$ partition edges $e_1, \ldots, e_n$ of $T$ into parts whose edge indices form intervals. The \tdef{ascent profile} of $T$, denoted by $\alf(T)$, is the sequence of lengths $(a_1, \ldots, a_r)$ of ascents in the left-to-right contour order. We note that $\firstasc(T) = a_1$ in this case. See \Cref{fig:tree-path-stat} for an illustration.

\subsection{The greedy Tamari poset on trees} \label{sec:greedy-tree}

For a plane tree $T \in \tree_n$, an \tdef{inner corner} of $T$ at a node $v$ is a corner at $v$ between two consecutive children of $v$, namely $v'$ and $v''$. It is clear from \Cref{defn:contour-path} that the inner corners of $T$ are in bijection with the valleys of its contour path $w_T$.

\begin{defn} \label{defn:greedy-poset-tree}
  For a plane tree $T \in \tree_n$ with a node $v''$ that is not the leftmost child of its parent $v$, we define the \tdef{slide-up} operation at $v''$ as follows:
  \begin{enumerate}
  \item Take $v'$ be the child of $v$ preceding $v''$, thus $(v, v', v'')$ form an inner corner;
  \item Detach edges of $v$ to subtrees rooted at $v''$ and all children of $v$ to its right;
  \item Slide these detached subtrees up along the edge $\{v, v'\}$ and reattach them to $v'$.
  \end{enumerate}
  See \Cref{fig:greedy-poset-tree} for an example of the slide-up operation. 
  
  By abuse of notation, denote by $\covgreedy$ the relation on $\tree_n$ such that $T \covgreedy T'$ if $T'$ can be obtained from $T$ by a slide-up. More precisely, write $T \covgreedy^{(i)} T'$ if $T'$ is obtained from $T'$ by a slide-up at the $i$-th node $u_i$ of $T$ in the pre-order. The relation $\covgreedy$ is the cover relation of an order $\leqgreedy$ and we denote by $(\tree_n, \leqgreedy)$ the corresponding poset.  See \Cref{fig:greedy-poset-tree} for an illustration of $(\tree_4, \leqgreedy)$.
\end{defn}

\begin{figure}
  \centering
  \insertfig{11}{1}
  \caption{Example of slide-up operation and the poset $(\tree_4, \leqgreedy)$.}
  \label{fig:greedy-poset-tree}
\end{figure}

We note that the slide-up at a vertex $v$ is possible if and only if $v$ is not the leftmost child of its parent. Our use of the same notation $\covgreedy$ for the orders on $\dyck_n$ and $\tree_n$ is justified by the following proposition.

\begin{prop} \label{prop:tree-greedy-poset}
  The poset $(\tree_n, \leqgreedy)$ is isomorphic to the greedy Tamari poset $(\dyck_n, \leqgreedy)$ by taking the contour path.
\end{prop}
\begin{proof}
  For $T, T' \in \tree_n$, let $w, w'$ be their contour paths respectively. We will show that $T \covgreedy T'$ if and only if $w \covgreedy w'$. Indeed, let $c = (v, v', v'')$ be an inner corner of $T$. The largest Dyck subpath $w^\circ$ obtained in \Cref{defn:greedy-poset} from the valley of $w$ corresponding to $c$ is given by the part of the contour walk of $T$ on all the subtrees of $v$ after the inner corner $c$, which form the part of $T$ which is reattached to obtain $T'$ in the slide-up at $v''$. As the effect of sliding up on the contour path is exactly exchanging $w^\circ$ with the $0$ that precedes it, which corresponds to the second visit to the edge $(v,v")$, we have the full equivalence.
\end{proof}

Using \cref{prop:tree-greedy-poset}, we will take $(\tree_n, \leqgreedy)$ as the greedy Tamari poset hereinafter. Note that the maximal element in $(\tree_n, \leqgreedy)$ is the unique tree of size $n$ with only one leaf, which we call the \tdef{linear tree} of size $n$. We now introduce the main object of study.

\begin{defn} \label{defn:greedy-tamari-interval}
  A \tdef{greedy Tamari interval} of size $n$ is a pair $I = (T, T')$ with $T, T' \in \tree_n$ such that $T \leqgreedy T'$. We denote by $\interval_n$ the set of greedy Tamari intervals in $(\tree_n, \leqgreedy)$. We define the \tdef{first ascent} (resp. \tdef{ascent profile}) of $I$ to be that of $T'$, \emph{i.e.}, $\firstasc(I) := \firstasc(T')$ (resp. $\alf(I) := \alf(T')$).
\end{defn}

Given a greedy Tamari interval $I = (T, T') \in \interval_n$, we define its weight as follows. Suppose that $\alf(T) = (a_1, \ldots, a_r)$, then we take
\begin{equation}
  \label{eq:interval-weight}
  \intweight(I) = x^{\firstasc(I)} \prod_{i = 2}^r p_{a_i}.
\end{equation}
The generating function $F_{\interval}(t, x) \equiv F_{\interval}(t, x; p_1, p_2, \ldots)$ is then defined by
\begin{equation}
  \label{eq:interval-gf}
  F_{\interval}(t, x) = \sum_{n \geq 0} t^n \sum_{I \in \interval_n} \intweight(I).
\end{equation}

\subsection{Rooted bipartite planar maps and their recursive decomposition} \label{sec:bipartite-map}

We now introduce the other side of our recursive bijection. A \tdef{planar map} is a drawing of a graph on the plane, defined up to continuous deformation, such that edges only intersect at vertices. The drawn edges cut the plane into \tdef{faces}, with a unique unbounded face called the \tdef{outer face}, while the others are the \tdef{inner faces}. We consider here only \tdef{rooted} planar maps, where one of the edges of the outer face is oriented clockwise. A \tdef{bipartite planar map} is a rooted planar map whose vertices are colored black and white such that each edge is adjacent to one vertex of each color, and with the root always oriented from a black vertex. See \Cref{fig:bipartite} for an example of a bipartite planar map. Denote by $\map_n$ the set of bipartite planar maps with $n$ edges.

\begin{figure}
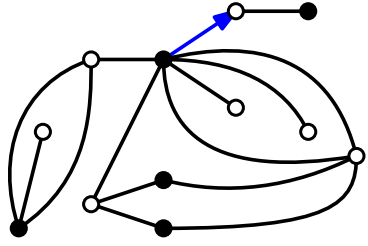

  \centering
  \insertfig{12}{0.3}
  \caption{Example of a (rooted) bipartite planar map}
  \label{fig:bipartite}
\end{figure}

Given a bipartite planar map $M$, we denote by $\outdeg(M)$ the half-degree of its outer face and by $\deg(f)$ the half-degree of a given inner face $f$. A generating function for bipartite planar maps $F_{\map}(t, x) \equiv F_{\map}(t, x; p_1, p_2, \ldots)$ is defined by
\begin{equation}
  \label{eq:map-gf}
  F_{\map}(t, x) = \sum_{n \geq 0} t^n \sum_{M \in \map_n} x^{\outdeg(M)} \prod_{f\text{ inner face of }M} p_{\deg(f)}.
\end{equation}

Tutte \cite{tutte-slicing} showed that 
\begin{equation}
  \label{eq:interval-decomp}
  F_{\map}(t, x) = 1 + xt F_{\map}(t, x)^2 + xt \Omega F_{\map}(t,x).
\end{equation}
Here, $\Omega$ is a linear operator in the space of polynomials in $x$, extended naturally to the space of power series in $t$ with coefficients polynomial in $x$, defined by
\begin{equation}
  \label{eq:omega-def}
  \forall i \geq 0,\quad \Omega x^i = \sum_{k=1}^i x^{i - k} p_k.
\end{equation}
\Cref{eq:interval-decomp}, given in \cite{tutte-slicing} by the name of ``even slicing'', can be explained combinatorially as follows. Taking out the edge next to the root in clockwise order of a non-empty bipartite planar map, there are two cases (see \cref{fig:bip-decomp}):
\begin{itemize}
\item The map breaks into two, in which case we root the two pieces at the appropriate edge adjacent to the vertices adjacent to the original root. To reconstruct a map in this case, we only need to take any pair of bipartite planar maps, add an edge from the black vertex of the root of the first component to the white vertex of the root of the second, which becomes the root.
\item The map remains connected then, by planarity, the deleted edge must separate an inner face and the outer face, thus the deletion adds the degree of the inner face to that of the outer face. To reconstruct a map in this case, we only need to take a bipartite planar map, then draw a new edge from the root to a white vertex to form a new inner face whose degree ``comes from'' that of the original outer face. The operator $\Omega$ describes every way to do that.
\end{itemize}
In both cases, an extra factor $xt$ accounts for the fact that there is a new edge on the outer face. This implies \cref{eq:interval-decomp}.

\begin{figure}
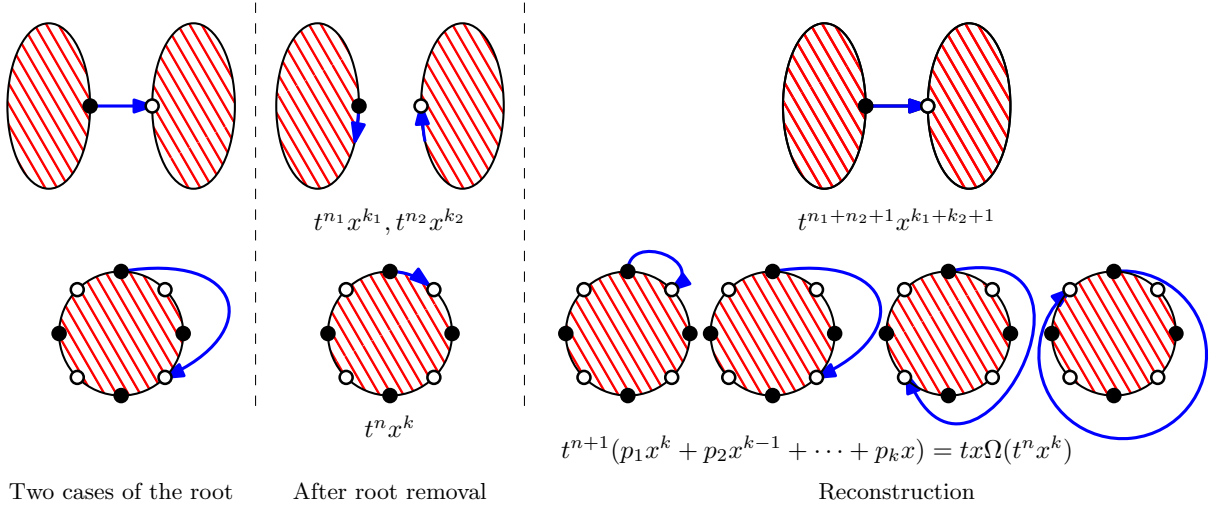

  \centering
  \insertfig{19}{1}
  \caption{Decomposition of bipartite planar maps}
  \label{fig:bip-decomp}
\end{figure}

We may consider \cref{eq:interval-decomp} as a generalization of the equation for the generating function of Catalan numbers, where the last term  is absent:
\[
  F_{\mathrm{Cat}}(t, x) = 1 + xt F_{\mathrm{Cat}}(t, x)^2.
\]
In this case, the terms in $F_{\mathrm{Cat}}$ are all of the form $(tx)^n$, and we can see from the decomposition above that they count rooted planar bipartite maps with no inner face, thus rooted plane trees, well-known to be counted by Catalan numbers.

Our goal is to give a recursive decomposition of greedy Tamari intervals isomorphic to that of bipartite planar maps illustrated in \Cref{fig:bip-decomp}, which will lead to a recursive bijection between greedy Tamari intervals and bipartite planar maps.

\section{Recursive decomposition of greedy Tamari intervals} \label{sec:decomp}

\subsection{Properties of greedy Tamari intervals} \label{sec:greedy-property}

We list a number of useful properties of the greedy Tamari intervals. We start with an important coalescence property of the greedy Tamari order.

\begin{lemma} \label{lem:alf-cover}
  For $T, T' \in \tree_n$ with $T \covgreedy T'$, either $\alf(T) = \alf(T')$, or $\alf(T')$ is obtained from $\alf(T)$ by merging two consecutive ascents, \emph{i.e.},
  \[
    \alf(T) = (a_1, \ldots, a_r),\quad \alf(T') = (a_1, \ldots, a_{s-1}, a_s + a_{s + 1}, a_{s + 2}, \ldots, a_r),
  \]
  for some $1 \leq s < r$. The second case occurs exactly when the edge along which the slide-up occurs is a leaf.
\end{lemma}
\begin{proof}
  Suppose that $T'$ is obtained from $T$ by a slide-up at the inner corner $(v, v', v'')$. Observe that the only ascent that may be affected is that of the leftmost leaf $\bar{v}$ of the slid-up subtree. When $v'$ is not a leaf in $T$, ascents are not changed, thus $\alf(T) = \alf(T')$. When $v'$ is a leaf in $T$, the ascent of $v'$ and $\bar{v}$ are merged in $T'$, which gives the merging of components in $\alf(T')$, since $\bar{v}$ is the next leaf after $v'$ in the pre-order of $T$.
\end{proof}

\begin{cor} \label{cor:parent-fixation}
  Let $I = (T, T') \in \interval_n$ be a greedy Tamari interval, with $u_0, \ldots u_n$ (resp. $u'_0, \ldots u'_n$) the nodes of $T$ (resp. $T'$) in the pre-order. For any $i$, if $u_{i+1}$ is the leftmost child of $u_i$ in $T$, then $u'_{i+1}$ is also the leftmost child of $u'_i$ in $T'$; if $u'_i$ is a leaf in $T'$, then $u_i$ is also a leaf in $T$.
\end{cor}
\begin{proof}
  We only need to prove the claim  for $T \covgreedy T'$. This follows from \Cref{lem:alf-cover} and the fact that, in a tree, $u_{i+1}$ is the leftmost child of $u_i$ if and only if they belong to the same ascent and $u_i$ is a leaf if and only if it is the last node in its ascent.
\end{proof}

\begin{defn} \label{defn:tree-portion}
  Given a tree $T\in \tree_n$  and $1 \leq k \leq n$, the \tdef{$k$-portion} of $T$, denoted by $\portion(T, k)$, is the tree induced by $T$ on the the first $k + 1$ nodes $u_0, u_1, \ldots, u_k$ in pre-order.
\end{defn}
Observe that for $l \leq k$ one has $\portion(\portion(T, k),l)=\portion(T, l)$.
The following is immediate.
\begin{lemma} \label{lem:portion-invariant}
  Let $T, T' \in \tree_n$ be such that $T \covgreedy^{(i)} T'$. Then for any $k$, one has $\portion(T, k) \covgreedy^{(i)} \portion(T', k)$ if $k\geq i$ and   $\portion(T, k) = \portion(T', k)$ if $k<i$.
\end{lemma}

Next we consider conditions under which one may switch the order of slide-ups in a chain in $(\tree_n, \leqgreedy)$.

\begin{lemma} \label{lem:switch-order}
Let $T, T', T'' \in \tree_n$ with $T \covgreedy^{(i)} T' \covgreedy^{(j)} T''$. If $i > j$, then the $j$-th node in $T$ is not the leftmost child of its parent, and the slide-up of this node leads to a new tree $T^*$ such that $T \covgreedy^{(j)} T^* \leqgreedy T''$.
\end{lemma}
\begin{proof}
  Denote by $u_i$ (resp. $u'_i$) the $i$-th node of $T$ (resp. $T'$) in the pre-order. Since $u'_j$ in $T'$ is not the leftmost child of its parent, by \cref{cor:parent-fixation}, $u_j$ is not the leftmost child of its parent either, meaning that we can indeed perform the slide-up operation at $u_j$ in $T$. This gives the existence of $T^*$.

  Now we show that $T^* \leqgreedy T''$. Observe that, if the subtree to be slid up at the corner just before $u_j$ in $T$ remains the same after sliding up the one for $u_i$ and \textit{vice versa}, then the two slide-ups commutes, and we have $T^* \leqgreedy T''$. Since $i > j$, there are two possibilities where we are not in such a case.
  \begin{itemize}
  \item The nodes $u_i$ and $u_j$ are siblings. In this case, we check easily that the two slide-up operations still commute, thus $T^* \covgreedy^{(i)} T''$.
  \item The node $u_i$ is the next sibling of the parent of $u_j$. In this case, as illustrated in \cref{fig:greedy-commute}, we can obtain $T''$ from $T^*$ by performing twice the slide-up operation at the corner just before the $i$-th node. We thus also have $T^* \leqgreedy T''$.
  \end{itemize}
  We thus conclude that we always have $T^* \leqgreedy T''$.
\end{proof}

\begin{figure}
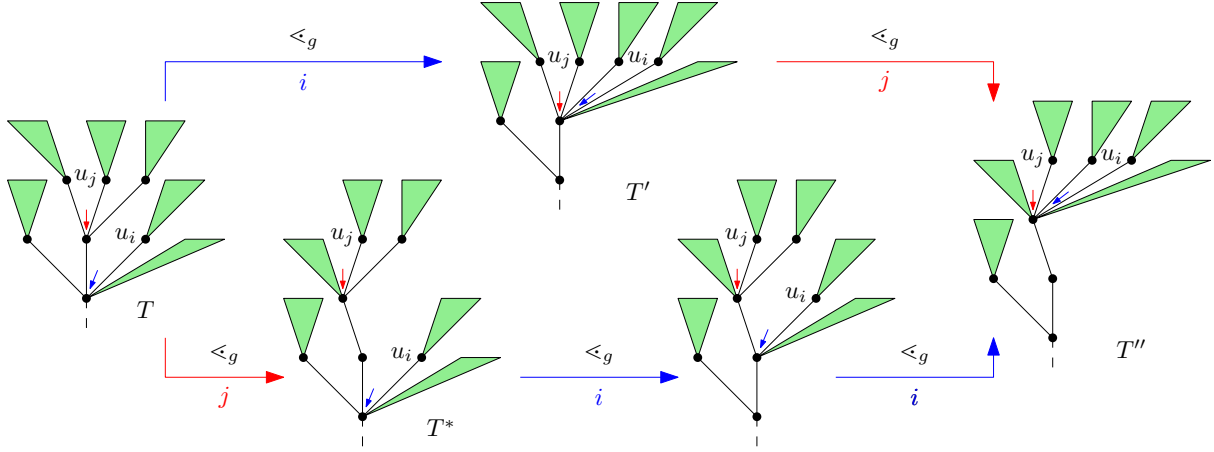

  \centering
  \insertfig{2}{1}
  \caption{The only non-trivial commutation of slide-up operations for cover relations in the greedy Tamari poset}
  \label{fig:greedy-commute}
\end{figure}

Given $I = (T, T') \in \interval_n$, \cref{lem:switch-order} provides a way to obtain a chain of maximal length from $T$ to $T'$ from any given chain.

\begin{cor} \label{cor:canonical-chain}
  Given $T, T' \in \tree_n$ with $T \leqgreedy T'$, there is a unique chain $C$ of maximal length
  \[
    T = T_0 \covgreedy^{(c_0)} T_1 \covgreedy^{(c_1)} \cdots \covgreedy^{(c_{k-1})} T_k = T'
  \]
  with $c_i$ nondecreasing. It is also the smallest in lexicographical order of the sequence $(c_0, \ldots, c_{k-1})$ among all chains from $T$ to $T'$.
\end{cor}
\begin{proof}
  For the existence of such a chain, it suffices to start from any chain $C'$ of maximal length from $T$ to $T'$ and apply successively \cref{lem:switch-order} until the sequence of the $c_i$'s is nondecreasing. Each such commutation makes the sequence of the $c_i$'s smaller lexicographically and possibly longer. However, as $C'$ is already of maximal length, the length stays the same at each commutation. As $(\tree, \leqgreedy)$ is a finite poset with a maximal element, the length of the chain is bounded, thus eventually \cref{lem:switch-order} will cease to be applicable, and we obtain a chain satisfying our conditions.

  To show that the chain is unique, suppose that there are two different chains satisfying our conditions, represented by the sequences $(c_i)_i$ and $(c'_i)_i$ of the inner corners used from $T$ to $T'$. Let $j$ be the smallest index with $c_j \neq c'_j$, and we may suppose that $c_j < c'_j$. By successive application of \cref{lem:portion-invariant} on $(c'_i)_{i \geq j}$, which is nondecreasing, we have $\portion(T_{j-1}, c_j) = \portion(T', c_j)$. However, the application of the same lemma on $(c_i)_{i \geq j}$ gives $\portion(T_{j-1}, c_j) \neq \portion(T', c_j)$, which is a contradiction. We thus conclude that the chain $C$ is unique, which also implies its minimality in lexicographical order, as one may start from any chain $C'$ from $T$ to $T'$ and apply \Cref{lem:switch-order} successively to obtain $C$, each step decreasing in lexicographical order.
\end{proof}

\begin{rem}
  Consider the labeling of edges in the cover relation $\covgreedy$ given by taking the label $i$ for all edges of the form $T \covgreedy^{(i)} T'$. By completing $(\tree_n, \leqgreedy)$ with a minimal element $\hat{0}$ with all cover edges labeled $0$, \Cref{cor:canonical-chain} implies that such a labeling is an EL-labeling, which ensures shellability of the corresponding order complex \cite{el-labeling}. We have also studied other aspects of the order structure of the greedy Tamari poset, which will be detailed in a forthcoming article.
\end{rem}

Given an interval $I \in \interval_n$, we define the \tdef{canonical chain} of $I$ to be the chain described in \cref{cor:canonical-chain}. In view of \cref{lem:portion-invariant,cor:canonical-chain}, we give the following construction that provides a special element in the canonical chain of $I$.

\begin{defn} \label{defn:straightening}
  Given a tree $T \in \tree_n$ and some $1 \leq k \leq n - 1$, the \tdef{$k$-straightening} of $T$, denoted by $\straight_k(T)$, is the tree  obtained as follows: starting from $T$, for each $i$ from $1$ to $k$, we perform repeatedly slide-ups at the $i$-th node of the current tree until no longer possible (it is possible that no slide-up is performed for a given $i$).
\end{defn}

Observe that $T = \straight_k(T)$ if and only if the first ascent of $T$ has length at least $k$, and every $\straight_k(T)$ is on the canonical chain from $T$ to the linear tree of the same size. Actually, one can be more precise, as the following proposition shows.

\begin{prop} \label{prop:straight}
  For a tree $T\in \tree_n$ and $k \leq n$, the tree $\straight_k(T)$ is on the canonical chain from $T$ to every $T^*$ with $T \leqgreedy T^*$ and $\firstasc(T^*) \geq k$.
\end{prop}
\begin{proof}
  Let $T_{\max}$ be the linear tree of size $n$, and $C$ (resp. $C^*$) the canonical chain from $T$ to $T_{\max}$ (resp. $T^*$). By \Cref{defn:straightening}, we have $\straight_k(T) \in C$. Let $T'$ be the largest tree in $C^* \cap C$, and $T_1$ (resp. $T_2$) its successor in $C$ (resp. $C^*$) with $T' \covgreedy^{(i)} T_1$ and $T' \covgreedy^{(j)} T_2$. By \Cref{cor:canonical-chain} on $C$, we have $i > j$, and combining \Cref{cor:canonical-chain,lem:portion-invariant} for trees after $T_2$ in $C^*$, we have $\portion(T^*, i) = \portion(T', i)$. However, as we have $T' \covgreedy^{(i)} T_1$, we know that $\portion(T', i)$ is not a linear tree, thus $k \leq \firstasc(T^*) < i$. Hence, $T'$ is above $\straight_k(T)$ in $C$ by \Cref{defn:straightening,cor:canonical-chain}, implying $\straight_k(T) \in C^*$. 
\end{proof}

\begin{defn} \label{defn:midpoint}
 Let $I = (T, T') \in \interval_n$, the \tdef{midpoint tree} of $I$, denoted by $\midtree(I)$, is the $k$-straightening of $T$ with $k = \firstasc(I)$, \emph{i.e.}, $\midtree(I) := \straight_k(T)$. See \cref{fig:midpoint-tree} for an example of the midpoint tree of a greedy Tamari interval.
\end{defn}

\begin{figure}
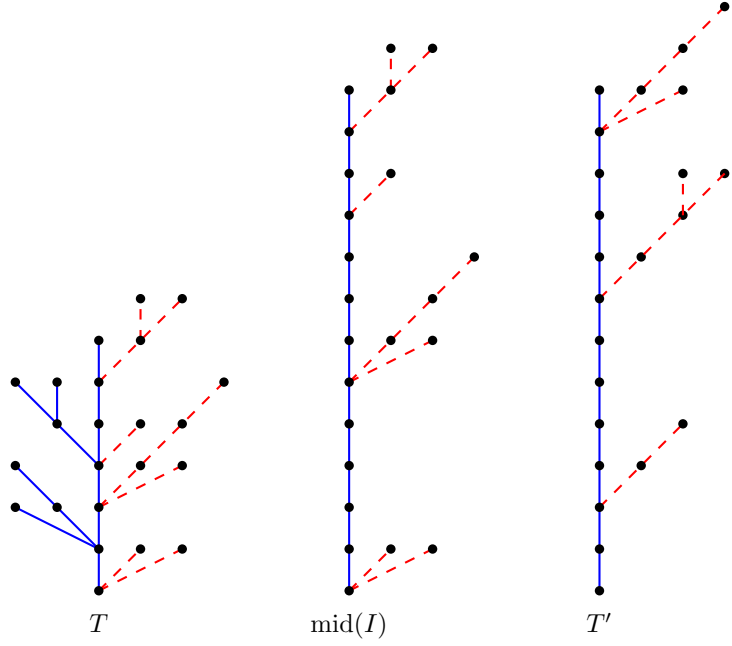

  \centering
  \insertfig{5}{0.6}
  \caption{Example of the midpoint tree of a greedy Tamari interval $I = (T, T')$, with backbone and side forest distinguished by colors}
  \label{fig:midpoint-tree}
\end{figure}

\begin{prop} \label{prop:midpoint-tree}
  For an interval $I = (T, T') \in \interval_n$, its midpoint tree $\midtree(I)$ is on its canonical chain and satisfies $\firstasc(\midtree(I))=\firstasc(I)$. Furthermore, the canonical chain from $\midtree(I)$ to $T'$ is a subchain of that of $I$.
\end{prop}
\begin{proof}
  This is a direct consequence of \Cref{prop:straight}.
\end{proof}

\begin{defn} \label{defn:backbone}
  For $I = (T, T') \in \interval_n$ and $R \in \tree_n$ with $T \leqgreedy R \leqgreedy T'$, the \tdef{backbone} of $R$ with respect to $I$ is $\portion(R,k)$ where $k = \firstasc(I)$. Let $u_0, \ldots, u_n$ be the nodes of $R$ in the pre-order. The \tdef{side-forest} of $R$ is the forest obtained from $R$ by deleting the edges of its backbone. This forest consists of  \tdef{side-trees} $\tau_i(R)$ rooted at the nodes $u_i$ for $0 \leq i < k$. We denote by $C_i(R)$ the set of indices of non-root nodes of $\tau_i(R)$, and by $c_i(R)$ the cardinality of $C_i(R)$. We draw backbones in blue and side-forests in red, see the examples in \Cref{fig:midpoint-tree}.
\end{defn}

Although the backbone of a tree $R$ is defined with respect to an interval $I$ containing $R$, in the following this interval should always be clear from the context. For an interval $I = (T, T')$, it is clear that the backbone of $\midtree(I)$ and of $T'$ is the linear tree of size $\firstasc(I)$. Note that, although one may consider $C_k(R)$ for $R$ inside $I$, by \Cref{cor:parent-fixation}, as $k = \firstasc(I) = \firstasc(T')$, the $k$-th node in pre-order of $T'$ is a leaf, hence also that of $R$, meaning that $C_k(R)$ is always empty.

Like ascents, the trees in the side-forest partition nodes not in the backbone into intervals in the pre-order. We define the \tdef{rightmost path} of a plane tree $T$ to be the path from the rightmost leaf of $T$ to the root.

\begin{lemma} \label{lem:side-straightening}
  For $I = (T, T') \in \interval_n$ and $R \in \tree_n$ with $T \leqgreedy R \leqgreedy T'$, let $k = \firstasc(R)$ and $u_0, \ldots, u_n$ the nodes of $R$ in pre-order. We have
  \begin{enumerate}
  \item The root $u_i$ is the only node in $\tau_i(R)$ with index at most $k$;
  \item All non-empty $C_i(R)$ are intervals of integers, and they partition $[k + 1, n]$.
  \item For $R = \midtree(I)$, every node $u_i$ with $C_i(R) \neq \varnothing$ is on the rightmost path of the backbone of $R$.
  \end{enumerate}
\end{lemma}
\begin{proof}
  The first point comes directly from \Cref{defn:backbone}, as every path linking $u_i, u_j$ with $i < j \leq k$ passes only through nodes preceding $u_j$ in the pre-order, thus in the backbone of $R$.

  The second point comes from the fact that every $u_j$ with $j > k$ is in some $\tau_i(R)$, and subtrees of $u_i$ in $\tau_i(R)$ are exactly those rooted at $u_j$ with $j > k$, which are consecutive subtrees of $u_i$, thus all their nodes come consecutively in the contour walk of $R$.

  The third point comes from the fact that, if a node $u_i$ is not on the rightmost path of the backbone $\portion(R, k)$ of $R$, then all its descendants are in $\portion(R, k)$.
\end{proof}

We note that the second point of \Cref{lem:side-straightening} shows that the sequence $(c_i(R))_{0 \leq i \leq k}$ determines all the sets $C_i(R)$'s. One can give a more precise description of the side-forest of the midpoint tree of an interval.

\begin{prop} \label{prop:midtree-side-pos}
  For $I = (T, T') \in \interval_n$, let $k = \firstasc(I)$ and $R = \midtree(I)$. We denote by $u_i$ the $i$-th node in pre-order of $T$. Then, for $0 \leq i < k$ with the side-tree $\tau_i(R)$ non-empty, the node $u_{i + 1}$ is on the rightmost path of the backbone of $T$.
\end{prop}
\begin{proof} 
  We observe in \Cref{defn:greedy-poset-tree} that, if $u_i$ and $u_j$ with $i < j$ are consecutive siblings in $T$, then they remain so for $T'$ with $T \covgreedy^{(\ell)} T'$ for $\ell < j$. Let $d = \min(C_i(R))$. By \Cref{lem:side-straightening} we have $d > k$. Suppose that $u_c$ is the sibling of $u_b$ immediately to its left in $T$. From our observation and \Cref{defn:midpoint}, they remain consecutive siblings in $R$, which implies $c = i + 1$ as $\portion(R, k)$ is a linear tree.
\end{proof}

An analogue of \cref{lem:alf-cover} holds for the side forest.

\begin{lemma} \label{lem:side-forest-moves}
  For $(T, T') \in \interval_n$ and $T \leqgreedy R \leqgreedy R' \leqgreedy T'$, the set partition formed by non-empty $C_i(R')$'s is refined by that formed by the non-empty $C_j(R)$'s.
\end{lemma}
\begin{proof}
  We first show our claim for $R \covgreedy^{(\ell)} R'$. We denotes by $u_i$ the $i$-th node of $R$ in pre-order, and we take $k = \firstasc(T')$. Let $u_a$ (resp. $u_b$) be the parent (resp. the preceding sibling) of $u_\ell$. We have $a < b < \ell$. There are several cases for $\ell$.
  \begin{itemize}
  \item Case $\ell \leq k$: In this case, $C_a(R') = \varnothing$, $C_b(R') = C_a(R)$, and $C_d(R') = C_d(R) = \varnothing$ for all $a < d < b$ by \Cref{prop:midtree-side-pos}, as no such $u_d$ can be on the rightmost path of the backbone of $R$ or $R'$. Hence, $C_i(R') = C_i(R)$ except for $i = a, b$, and the concerned set partitions are the same for both $R$ and $R'$.
  \item Case $\ell > k$: If $b > k$, then $u_b$ is also in some $\tau_j(R)$, thus $C_i(R') = C_i(R)$ for all $i$. Otherwise, $u_b$ must be on the rightmost path of the backbone of $R$, thus $C_a(R') = \varnothing$, $C_b(R') = C_a(R) \cup C_b(R)$. By the same argument as in the previous case, we have $C_d(R') = C_d(R) = \varnothing$ for $a < d < b$. Hence, the concerned set partition for $R'$ are either the same as that for $R$, when $C_a(R)$ or $C_b(R)$ is empty, or is otherwise obtained from that for $R$ by merging two consecutive non-empty parts $C_a(R)$ and $C_b(R)$.
  \end{itemize}
  We then conclude by observing that our claim is transitive.
\end{proof}

Note that, while \cref{lem:alf-cover} holds also for the classical Tamari order, \Cref{lem:side-forest-moves} does not. This is a place where the crucial difference between the two orders manifests itself.

\subsection{Construction of greedy Tamari intervals}

We introduce two constructions of larger greedy Tamari intervals from smaller ones, whose combination shows that the generating function $F_\interval$ of greedy Tamari intervals also satisfies \Cref{eq:interval-decomp}, leading to a recursive bijection transferring the corresponding statistics used to define the weights given in \Cref{eq:interval-weight}.

\begin{defn} \label{defn:addleaf-constr}
  Given $I = (T, T') \in \interval_n$, for $0 \leq k \leq n$, we define $\addleaf(I, k) = (S, S')$ as follows. Let $u_0, \ldots, u_n$ (resp. $u'_0, \ldots, u'_n$) be the nodes of $T$ (resp. $T'$) in the pre-order. Then $S$ (resp. $S'$) is obtained from $T$ (resp. $T'$) by inserting a leaf as the leftmost child of $u_k$ (resp. $u'_k$). 
\end{defn}

See \cref{fig:addleaf-constr} for an example of the construction of $\addleaf(I, k)$. It is clear that $\addleaf(I, k)$ is a greedy Tamari interval of size $n + 1$, as there is a new edge, and the slide-ups transforming $T$ to $T'$, up to shifting those at the nodes after the $k$-th node in the pre-order by 1, also transform $S$ to $S'$.

\begin{figure}
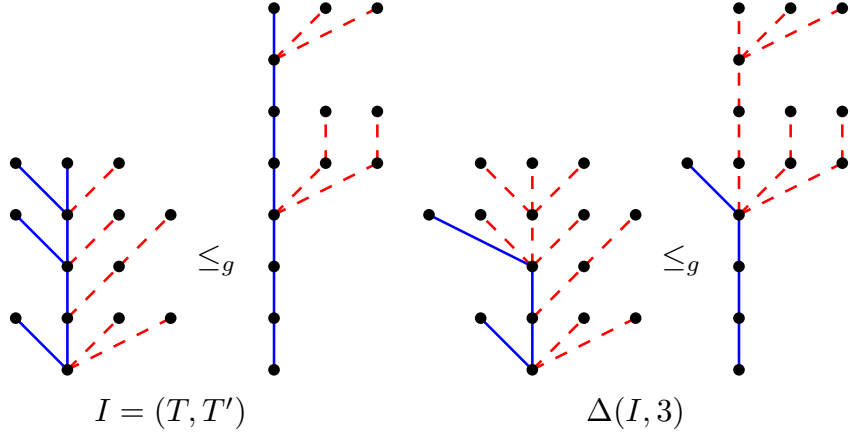

  \centering
  \insertfig{6}{0.7}
  \caption{Example of the first construction $\addleaf(I, k)$}
  \label{fig:addleaf-constr}
\end{figure}

\begin{prop} \label{prop:addleaf-weight}
  For $I = (T, T') \in \interval_n$, let $\ell = \firstasc(I)$. For $0 \leq k < \ell$, we have
  \[
    \intweight(\addleaf(I, k)) = tx^{k+1-\ell}p_{\ell - k} \intweight(I).
  \]
  We also have
  \[
    \intweight(\addleaf(I, \firstasc(I))) = tx \intweight(I).
  \]
\end{prop}
\begin{proof}
  Let $(S, S') = \addleaf(I, k)$, and $u'_0, \ldots, u'_n$ the nodes of $T'$ in the pre-order. Since $0 \leq k \leq \ell$, the node $u'_k$ is on the first ascent of $T$, and $u'_k$ is an internal node when $k < \ell$.

  For the first part, $S'$ is obtained by adding a leaf $v$ as the leftmost child of $u_k$ in $T'$, which makes $v$ the first leaf of $S'$. We thus have $\firstasc(S') = k + 1$. The node $u'_\ell$ in $T'$, which is a leaf, has its ascent  cut off at $u'_k$, thus adding a factor $p_{\ell - k}$. This gives the weight of $(S, S')$. For the second part, the new leaf $v$ is added as the child of $u'_\ell$, thus replaces $u'_\ell$ as the leftmost leaf. The first ascent is thus extended by $1$ without introducing any new leaf. Accounting the changes in weights gives the result.
\end{proof}

The second construction will produce an interval $I = (S, S')$ from two smaller intervals $I_1 = (T_1, T'_1)$ and $I_2 = (T_2, T'_2)$. Roughly, $S'$ will be simply obtained by putting $T'_2$ above $T'_1$, and for $S$ we will need to cut $T_1$ into two parts and insert $T_2$ inbetween. We now detail the precise construction.

Given $I = (T, T') \in \interval_n$ with $n > 0$, let $k = \firstasc(I)$. Let $u_\ell$ be the lowest node on the rightmost path $P$ of the backbone $\portion(k, T)$ of $T$ such that for every node $u_a$ in $P$ above it, including itself, one has $C_a(T) = \varnothing$. The existence of $u_\ell$ is ensured by $C_k(T) = \varnothing$ and $u_k \in P$. We define the \tdef{head} of $T$ in the interval $I$ to be the subtree of $T$ rooted at $u_\ell$, which is also a subtree of $\portion(T, k)$, and the \tdef{body} of $T$ to be the tree obtained from $T$ by cutting out its head. See \cref{fig:head-body-decomp} for an illustration. Note that the body of $T$ may be empty in the sense of not containing any node, which corresponds to the case where $T'$ is the maximal element of the greedy Tamari poset of order $n$, in which case the backbone of $T$ is $T$ itself, thus $u_\ell = u_0$. However, the head of $T$ is never empty, its number of nodes is $h = k - \ell + 1 \geq 1$.

\begin{figure}
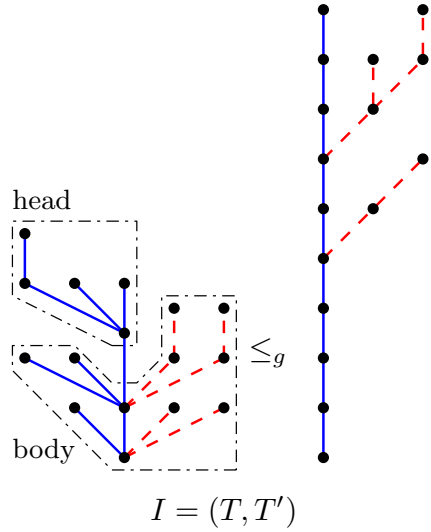

  \centering
  \insertfig{8}{0.35}
  \caption{An example of the head and body of $T$ in a greedy Tamari interval $I=(T, T')$}
  \label{fig:head-body-decomp}
\end{figure}

\begin{defn} \label{defn:product-constr}
  Let $I_1 = (T_1, T'_1)$ and $I_2 = (T_2, T'_2)$ be greedy Tamari intervals, with $n_1 \geq 0$ the size of $I_1$ and $n_2 > 0$ the size of $I_2$. We define $\intproduct(I_1, I_2) = (S, S')$ as follows, illustrated in \cref{fig:product-constr}. The tree $S'$ is obtained by attaching $T'_2$ as a subtree of the leftmost leaf of $T'_1$, For the tree $S$, let $k_2 = \firstasc(I_2)$. We first insert the head of $T_1$ as a subtree in $T_2$ so that its root is the sibling just to the left of the $k_2$-th node of $T_2$ in the pre-order, and we denote by $C_*$ the tree obtained. Then $S$ is obtained by replacing in $T_1$ its head by $C_*$. It is clear that both $S$ and $S'$ are trees of size $n_1 + n_2 + 1$.
\end{defn}

\begin{figure}
  \centering
  \insertfig{7}{0.8}
  \caption{An example of the construction $\intproduct(I_1, I_2)$}
  \label{fig:product-constr}
\end{figure}

\begin{prop} \label{prop:product-weight}
  For $I_1 = (T_1, T'_1) \in \interval_{n_1}$ and $I_2 = (T_2, T'_2) \in \interval_{n_2}$ with $n_1 \geq 0$ and $n_2 > 0$, let $(S, S') = \intproduct(I_1, I_2)$. Then $J = (S, S') \in \interval_{n_1 + n_2 + 1}$. Furthermore, one has
\[
  \intweight(J) = tx \intweight(I_1) \intweight(I_2).
\]
\end{prop}
\begin{proof}
  We first show that $J$ is indeed an interval by showing a way to go from $S$ to $S'$ using slide-ups. Let $u_0, \ldots, u_{n_1}$ (resp. $v_0, \ldots, v_{n_2}$) be the nodes in $T_1$ (resp. $T_2$) in pre-order. For simplicity, we identify nodes in $T_1$ and $T_2$ with their counterpart in $S$ from the construction of $S$ in \cref{defn:product-constr}.

  We take $k_1 = \firstasc(I_1)$, $k_2 = \firstasc(I_2)$, and $k = k_1 + k_2 + 1$. From the construction of $S'$ in \cref{defn:product-constr}, the leftmost leaf in $S'$ is the $k$-th node in the pre-order. Let $S_{\midtree} = \straight_k(S)$ be the $k$-straightening of $S$. By definition, we have $S \leqgreedy S_{\midtree}$, and $\portion(S_{\midtree}, k)$ is the linear tree  of size $k$. Let $S_+$ be the tree obtained by attaching $\midtree(I_2)$ as a subtree to the $k_1$-th node of $\midtree(I_1)$, which is a leaf. It is clear that both $S_{\midtree}$ and $S_+$ have the same $k$-portion, and they also have the same subtrees attached to this $k$-portion, because the parental relations for side trees remains the same in the midpoint trees. The side trees of $S_{\midtree}$ coming from $T_1$ are already at the same positions as in $S_+$, while the side trees of $S_{\midtree}$ coming from $T_2$ have their points of attachment to the backbone  lower than the corresponding ones in $S_+$ by the size of the head of $T_1$ plus one. It follows that $S_{\midtree} \leqgreedy S_+$, since we can move the side trees of $S_{\midtree}$ from $S_2$ up to their positions in $S_+$ by successive slide-ups, for side trees from above to below, sliding up all the side trees attached to the same node at the corner before all of them. Now, it is clear that $S_+ \leqgreedy S'$, since we may apply the slide-ups used to get from $\midtree(I_1)$ to $T'_1$ and those from $\midtree(I_2)$ to $T'_2$ to the part of $S_+$ from $\midtree(I_1)$ and $\midtree(I_2)$ respectively to obtain $S'$ without any interference between the two parts. We thus conclude that $S \leqgreedy S_{\midtree} \leqgreedy S_+ \leqgreedy S'$, so that $J = (S, S')$ is indeed an interval.

  For the weights, we simply check that the size of $J$ is the sum of the sizes of $I_1$ and $I_2$ plus one, that $\firstasc(J) = \firstasc(I_1) + \firstasc(I_2)$, and that the side trees of $S'$ are exactly those in $T'_1$ and $T'_2$, thus keeping the ascents of all leaves in $T'_1$ and $T'_2$ except the first ones.
\end{proof}

\subsection{Decomposition of greedy Tamari intervals}

Given a greedy Tamari interval of size $n > 0$, we show that it either takes the form $\addleaf(I, k)$ or $\intproduct(I_1, I_2)$.

\begin{theorem} \label{thm:main-decomp}
  Let $J = (S, S')$ be a greedy Tamari interval of size $n > 0$ and $k = \firstasc(J)$. Then, depending on the situation of the $k$-th node $v_k$ of $S$ in the pre-order, $J$ can be uniquely represented in one of the following two forms:
  \begin{enumerate}
  \item \textbf{Type 1}: If $v_k$ is the leftmost child of its parent in $S$, then $J = \addleaf(I, k - 1)$ for some unique $I \in \interval_{n - 1}$;
  \item \textbf{Type 2}: Otherwise, we have $J = \intproduct(I_1, I_2)$ for some unique pair $I_1 \in \interval_{n_1}$, $I_2 \in \interval_{n_2}$ with $n_1 \geq 0$, $n_2 > 0$ and $n_1 + n_2 + 1 = n$.
  \end{enumerate}
\end{theorem}
\begin{proof}
  Let $v_0, \ldots, v_n$ (resp. $v'_0, \ldots, v'_n$) be the nodes of $S$ (resp. $S'$) in the pre-order. Since $k = \firstasc(J)$, the $k$-th node $v'_k$ of $S'$ is a leaf. It folllows from \cref{cor:parent-fixation} that $v_k$ is also a leaf.

  Suppose that $J$ is of type 1, let $T$ (resp. $T'$) be the plane tree of size $n - 1$ obtained by removing $v_k$ (resp. $v'_k$) from $S$ (resp. $S'$). For simplicity, we identify nodes in $T$ (resp. $T'$) with their counterpart in $S$ (resp. $S'$).  Since $v'_k$ is  a leaf in $S'$,   there is never a slide-up at the corner just after $v_k$, thus the chain of slide-ups that turns $S$ into $S'$ also turns $T$ into $T'$ therefore $T \leqgreedy T'$. Let $I = (T, T') \in \interval_{n-1}$. Since $J$ is of type 1, so that $v_k$ is the leftmost child of its parent $v_{k-1}$ in $S$, one has $J = \addleaf(I, k-1)$. If there is another pair $I', k'$ such that $J = \addleaf(I', k')$, then by \cref{prop:addleaf-weight}, we have $k' = k - 1$, and thus $I' = I$ as we are adding leaves at the same positions of trees of $I$ and $I'$ to obtain the same trees in $J$. This implies the uniqueness of the pair $I,k$.

  If $J$ is of type 2, then $v_k$ has a sibling $u$ preceding it. Denote by $H$ the subtree of $S$ rooted at $u$, and by $h$ the number of nodes in $H$. It is clear that all the nodes of $H$ precede $v_k$ in the pre-order, hence $H$ is in the backbone of $S$. In order to  decompose $I$ into $\intproduct(I_1, I_2)$ with $I_1 = (T_1, T'_1)$ and $I_2 = (T_2, T'_2)$, by \Cref{defn:product-constr}, $S'$ must be obtained by attaching $T'_2$ to the leftmost leaf of $T'_1$. Since $n_1 \geq 0$ and $n_2 > 0$, there is thus an index $0 \leq j_* < k - 1$ such that $v'_{j_*}$ in $S'$ is the leftmost leaf of $T'_1$, therefore $T_1', T_2'$ are obtained by cutting the edge between $v'_{j_*}$ and $v'_{j_*+1}$ in $S'$. Note that the backbone of $T_1$ must have $j_* + 1$ nodes and contain $H$, leaving $j_* + 1 - h$ nodes not in $H$. We then obtain $T_1$ by cutting the edge from $v_{j_* - h + 1}$ to its parent, which ensures the correct number of nodes in the backbone of $T_1$, and attach $H$ to the former parent of $v_{j_* - h + 1}$ as the rightmost subtree. When $v_{j_* - h + 1}$ has no parent, then $T_1 = H$. Note that the construction is invalid if $v_{j_* - h + 1}$ is not on the rightmost path of the backbone of $S$. The tree $T_2$ is the rest of $S$ after the construction of $T_1$. See \Cref{fig:int-decomp-type-2} for an example. It is clear that not every index $j_*$ is valid, and the validity of $j_*$ is given by the following conditions:
  \begin{enumerate}[label=\textbf{(C\arabic*)}]
  \item \textbf{Cutting at leaf}: The node $v'_{j_*}$ on $S'$ should have exactly one child. This is to ensure that $v'_{j_*}$ gives the last node of the backbone of $T'_1$, which must be a leaf. \label{cond:cut-at-leaf}
  \item \textbf{Head-body cutting}: The node $v_{j_* - h + 1}$ should be on the rightmost path of the backbone of $S$, and its parent, when it exists, should have a non-empty side-tree. This is to ensure that our construction of $T_1$ is the inverse of its head-body decomposition. \label{cond:cut-head-body}
  \item \textbf{Size-matching}: $T_1$ and $T'_1$ should have the same size, which is equivalent to $T_2$ and $T'_2$ having the same size. \label{cond:size-match}
  \item \textbf{Comparability}: We should have $T_1 \leqgreedy T'_1$ and $T_2 \leqgreedy T'_2$. \label{cond:comparable}
  \end{enumerate}
  It is clear that an index $j_*$ with $0 \leq j < k - 1$ gives a valid decomposition $J = \intproduct(I_1, I_2)$ if and only if it satisfies the conditions above.

  Let us now  analyze each condition in terms of the side-forests of $S$ and $S'$. Recall from \Cref{defn:backbone} that, for $R$ in the interval $I$, the number of non-root nodes in the side-tree $\tau_i(R)$ rooted at the $i$-th node of $R$ in the pre-order is denoted by $c_i(R)$. Let $S_{\midtree}$ be the midpoint tree of $J$.
  \begin{itemize}
  \item Condition \Cref{cond:cut-at-leaf} is equivalent to $\tau_{j_*}(S')$ being empty, thus $c_{j_*}(S') = 0$.
  \item Condition \Cref{cond:cut-head-body} is equivalent to $c_{j_* - h}(S_{\midtree}) > 0$ when $j_* \geq h$ by \Cref{defn:midpoint,lem:side-forest-moves}. We note that this imposes $j_* \geq h - 1$, as the only case where $j_* < h$ is when $v_{j_* - h + 1}$ is the root of $S$, meaning that $j_* = h - 1$.
  \item Condition \Cref{cond:size-match} is equivalent to the equality of total size of the side-forests of $T_1, T'_1$, since they have the same number of nodes in the backbone. If we further suppose \Cref{cond:cut-head-body}, since $v_{j_* - h + 1}$ is on the rightmost path of the backbone of $S$, by \Cref{defn:midpoint}, the condition \Cref{cond:size-match} is equivalent to
    \begin{equation}
      \label{eq:cond-size-match}
      \sum_{i=0}^{j_* - h} c_i(S_{\midtree}) = \sum_{i=0}^{j_*} c_i(S').
    \end{equation}
  \item For \Cref{cond:comparable}, we first observe that $\firstasc(T'_1) = j_*$ and $\firstasc(T'_2) = k - 1 - j_*$. Let $R_1 = \straight_{j_*}(T_1)$ be the $j_*$-straightening of $T_1$, and $R_2 = \straight_{k - 1 - j_*}(T_2)$ the $(k - 1 - j_*)$-straightening of $T_2$. By \Cref{prop:midpoint-tree}, $T_1 \leqgreedy T'_1$ if and only if $R_1 \leqgreedy T'_1$, and $T_2 \leqgreedy T'_2$ if and only if $R_2 \leqgreedy T'_2$.

    Let us now consider the structure of $R_1$ and $R_2$. By the definition of head and tail of $T_1$, the last $h$ nodes of the backbone of $R_1$ form a linear tree. By \Cref{defn:midpoint}, similar to the decomposition of $S$, the midpoint tree $S_{\midtree}$ of $J$ can be obtained by first replacing the last $h$ nodes of $R_1$ by $R_2$, then attaching the linear tree of these $h$ nodes on top of the last leaf of the backbone of $R_2$. See \Cref{fig:int-decomp-type-2} for an example. Hence, we have
    \[
      \tau_i(S_{\midtree}) =
      \begin{cases}
        \tau_i(R_1) &\text{ for } 0 \leq i \leq j_* - h; \\
        \tau_{i - j_* + h - 1}(R_2) &\text{ for } j_* - h + 1 \leq i \leq k - h; \\
        \varnothing &\text{ for } k - h + 1 \leq i \leq k.
      \end{cases}
    \]
    Since $S_{\midtree} \leqgreedy S'$, we see that the slide-up moves on the canonical path from $S_{\midtree}$ to $S'$ restricted to nodes in $\tau_i(S_{\midtree})$ with $i \leq j_* - h$ gives us a path from $R_1$ to $T'_1$, hence we always have $T_1 \leqgreedy T'_1$. For $R_2$, we only need to ensure that every side-tree of $R_2$ is lower than the corresponding part of side-tree of $T'_2$, as the slide-up moves on the canonical path from $S_{\midtree}$ to $S'$ not changing side-tree positions remain valid from $R_2$ to $T'_2$. By \Cref{lem:side-straightening,lem:side-forest-moves}, we have $R_2 \leq T'_2$ if and only if $\sum_{i = 0}^{i'} c_{i}(R_2) \geq \sum_{i = 0}^{i'} c_{i}(T'_2)$ for all $i' \geq k - 2 - j_*$, as $c_{k - 1 - j_*}(R_2) = c_{k - 1 - j_*}(T'_2) = 0$. By the expression of $c_i(S_{\midtree})$ and the construction of $T_2$, with the change of variable $j' = i' + j_* + 1$, this condition is equivalent to
    \[
      \sum_{i = j_* - h + 1}^{j' - h} c_{i}(S_{\midtree}) \geq \sum_{i = j_* + 1}^{j'} c_{i}(S') \text{ for all } j_* < j' \leq k - 2.
    \]
    If we further suppose \Cref{cond:cut-head-body,cond:size-match}, hence \Cref{eq:cond-size-match}, we have
    \begin{equation}
      \label{eq:cond-comparable}
      \sum_{i = 0}^{j' - h} c_{i}(S_{\midtree}) \geq \sum_{i = 0}^{j'} c_{i}(S') \text{ for all } j_* < j' \leq k - 2.
    \end{equation}
  \end{itemize}

  \begin{figure}
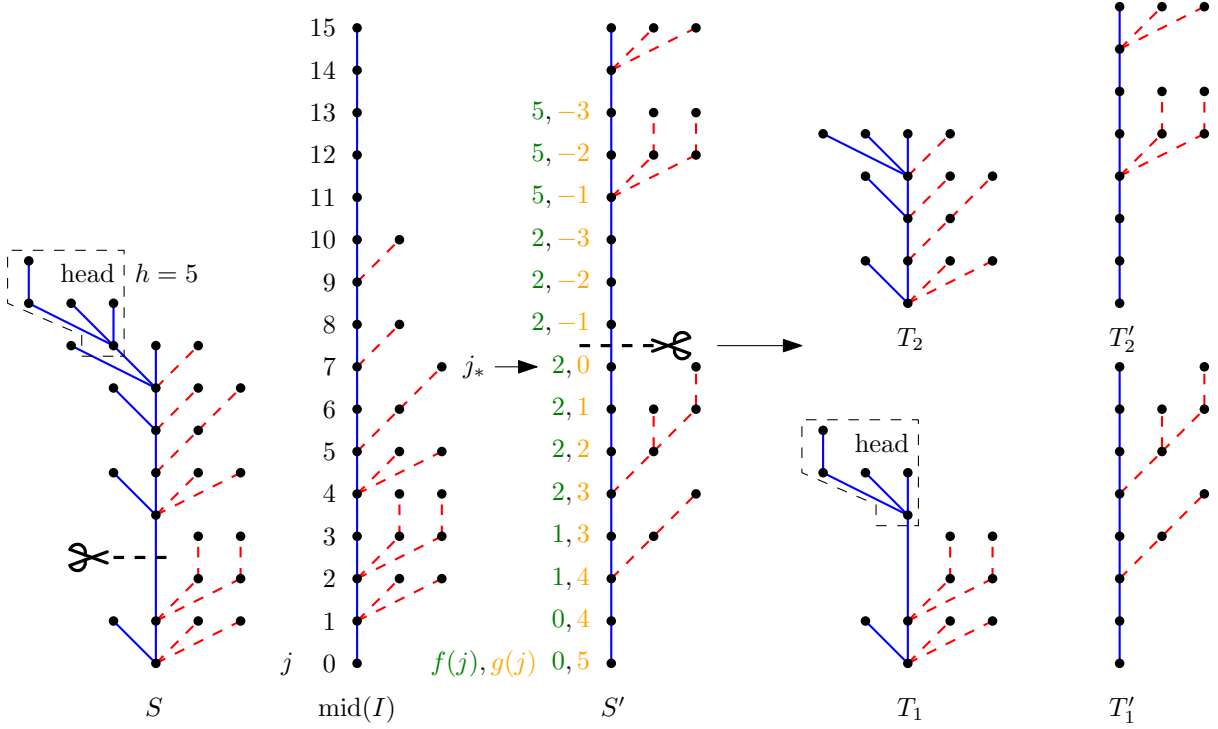

    \centering
    \insertfig{13}{1}
    \caption{Example of the decomposition of a type-2 greedy Tamari interval}
    \label{fig:int-decomp-type-2}
  \end{figure}
  
  Motivated by \Cref{eq:cond-size-match,eq:cond-comparable}, we define the function $f$ by
  \begin{equation}
    \label{eq:f-def}
    f(j) = \min \left\{ m \geq -1 \; \middle| \; \sum_{i = 0}^m c_i(S_{\midtree}) = \sum_{i = 0}^{j} c_i(S') \right\}.
  \end{equation}
  Here, we take the convention that the empty summation for $m = -1$ is zero, and we take $f(-1) = -1$. Equivalently, $f(j)$ is the root position of the highest side-tree of $S_{\midtree}$ that is part of subtrees of $S'$ rooted at or below $v_j$, and is $-1$ when the set of such side-trees is empty. See \Cref{fig:int-decomp-type-2} for an example of $f(j)$. By \Cref{lem:side-forest-moves}, the set in \Cref{eq:f-def} is never empty, hence $f(j)$ is well-defined. By \Cref{eq:f-def}, the function $f$ is non-decreasing, and $f(j) > f(j - 1)$ if and only if $c_j(S') > 0$. Note that this holds for $j = 0$, as $f(0) = -1$ if and only if $c_0(S') = 0$.
  
  We also define the function $g(j)$ by
  \begin{equation}
    \label{eq:g-def}
    g(j) = f(j) - j + h.
  \end{equation}
  Note that \Cref{eq:cond-size-match} is satisfied for $j_*$ when $g(j_*) = 0$. We may also express the other conditions in terms of $f$ and $g$. Since $f$ is non-decreasing, one has $g(j) - g(j - 1) \geq -1$ and equality holds exactly when $c_j(S') > 0$. Since $g(-1) = f(-1) + h > 0$ (indeed, $f(-1) = -1$ and $h > 0$) and $g(k - 2) = f(k - 2) - k + 2 - h \leq 0$ as $f(k - 2) \leq k - 1$ by definition, there is an index $0 \leq j' \leq k - 2$ such that $g(j') = 0$ and $g(j') - g(j' - 1) = -1$. This is equivalent to $f(j') = j' - h$ and $f(j') = f(j' - 1)$. This condition is equivalent to \Cref{cond:cut-at-leaf}, \Cref{cond:cut-head-body} and \Cref{cond:size-match} combined:
  \begin{itemize}
  \item \Cref{cond:cut-at-leaf} is equivalent to $c_{j_*}(S') = 0$, thus also to $f(j') = f(j' - 1)$;
  \item \Cref{cond:cut-head-body} and \Cref{cond:size-match} is equivalent to the definition of $f(j_*)$ by \Cref{eq:f-def}. More precisely, \Cref{cond:cut-head-body}, which is equivalent to $c_{j_* - h}(S_{\midtree}) > 0$, corresponds to the minimality of $f(j_*)$, while \Cref{eq:cond-size-match} for \Cref{cond:size-match} is precisely the equality condition of the set in \Cref{eq:f-def}.
  \end{itemize}

  Now we show that among all indices $0 \leq j' \leq k - 2$ satisfying $f(j') = j' - h$ and $f(j') = f(j' - 1)$, only the largest one $j_*$ satisfies \Cref{cond:comparable}. Indeed, for $j' < j_*$, there must be some $j' \leq j'' < j_*$ such that $g(j'') > 0$, hence $f(j'') > j'' - h$. By \Cref{eq:f-def}, we have
  \[
    \sum_{i = 0}^{j'' - h} c_i(S_{\midtree}) < \sum_{i = 0}^{f(j'')} c_i(S_{\midtree}) = \sum_{i = 0}^{j''} c_i(S').
  \]
  The strict inequality comes from the minimality of $f(j'')$ in \Cref{eq:f-def}. Thus, \Cref{eq:cond-comparable} is violated by $j''$, meaning that $j'$ does not satisfy \Cref{cond:comparable}. Now we show that $j_*$ indeed satisfies \Cref{cond:comparable}. By the maximality of $j_*$, for every $j$ with $j_* < j \leq k - 2$, we have $g(j) \leq 0$, meaning that $f(j) \leq j' - h$. As $f$ is increasing, we thus have
    \[
      \sum_{i = 0}^{j' - h} c_i(S_{\midtree}) \geq \sum_{i = 0}^{f(j)} c_i(S_{\midtree}) = \sum_{i = 0}^{j} c_i(S').
    \]
    Hence, \Cref{eq:cond-comparable} holds for $j_*$. As \Cref{cond:cut-head-body,cond:size-match} hold for $j_*$, \Cref{cond:comparable} also holds for $j_*$. With $j_*$ the only valid index for all four conditions, we have the unique decomposition of $J$ of type $2$.
\end{proof}

\begin{theorem} \label{thm:main-func-eq}
  The generating function $F_{\interval}$ of greedy Tamari intervals satisfies
  \[
    F_{\interval} = 1 + xt F_{\interval}(t,x)^2 + xt \Omega F_{\interval}(t,x).
  \]
  We thus have $F_{\interval}(t,x) = F_{\mathcal M}(t,x)$, where $F_{\mathcal M}(t,x)$ the generating function of bipartite planar maps, which satisfies \Cref{eq:interval-decomp}.
\end{theorem}
\begin{proof}
  We look at the contribution of different types of intervals to $F_{\interval}(t, x)$.
  \begin{itemize}
  \item The interval of size $0$: $1$.
  \item Type 1 intervals in \cref{thm:main-decomp}: $xt\Omega F_{\interval}(t,x) + xt F_{\interval}(t,x)$, by \cref{prop:addleaf-weight}.
  \item Type 1 intervals in \cref{thm:main-decomp}: $xt F_{\interval}(t,x) (F_{\interval}(t,x) - 1)$, by \cref{prop:product-weight} and the fact that $I_2$ may not be of size $0$ in $\intproduct(I_1, I_2)$.
  \end{itemize}
  All the contributions add up to the right-hand side of our claimed equation, which is the same as \cref{eq:interval-decomp} for $F_{\mathcal{C}^{(2)}}(t,x)$. As both $F_{\interval}$ and $F_{\mathcal{C}^{(2)}}$ are power series in $t$ with coefficients polynomial in $x$, they are totally determined by \cref{eq:interval-decomp}, meaning that they are equal.
\end{proof}

Comparing \Cref{thm:main-func-eq} with \Cref{eq:interval-decomp}, up to rearrangement of terms, we see that the recursive decomposition of greedy Tamari intervals in \Cref{thm:main-decomp} is isomorphic to that of bipartite planar maps illustrated in \Cref{fig:bip-decomp}. We thus settle \cite[Conjecture~6.1]{greedy-tamari-interval} for the case $m = 1$.

\section{Greedy $m$-Tamari intervals and $(m+1)$-constellations} \label{sec:m-decomp}

\begin{defn} \label{defn:m-ballot}
  Given $m \geq 1$, an \tdef{$m$-ballot path} is a Dyck path such that the lengths of its ascents are multiples of $m$. The length of the Dyck path is a multiple of $m$, say $mn$ and we call $n$ its size and denote by $\mdyck_n \subseteq \dyck_{mn}$ the set of $m$-ballot paths of size $n$. The special case of $1$-ballot paths is simply Dyck paths.
\end{defn}

Similarly, by taking the contour path we have a corresponding notion of trees.

\begin{defn} \label{defn:m-tree}
  An \tdef{$m$-branch tree} is a rooted plane tree such that, for every leaf, the length of its ascent is a multiple of $m$. We denote by $\mtree_n$ the set of $m$-branch trees with $mn$ edges, and one has $\mtree_n \subseteq \tree_{mn}$. The size of an $m$-branch tree is its number of edges divided by $m$.
\end{defn}

It follows from \cref{lem:alf-cover} that, if $T$ is an $m$-branch tree and $T \covgreedy T'$, then $T'$ is also an $m$-branch tree. See \Cref{fig:greedy-cover} for an example with $m = 2$ and $n = 9$. Hence, $\mtree_n$ is an order filter of $(\tree_{mn}, \leqgreedy)$. We therefore have the following definition.

\begin{figure}
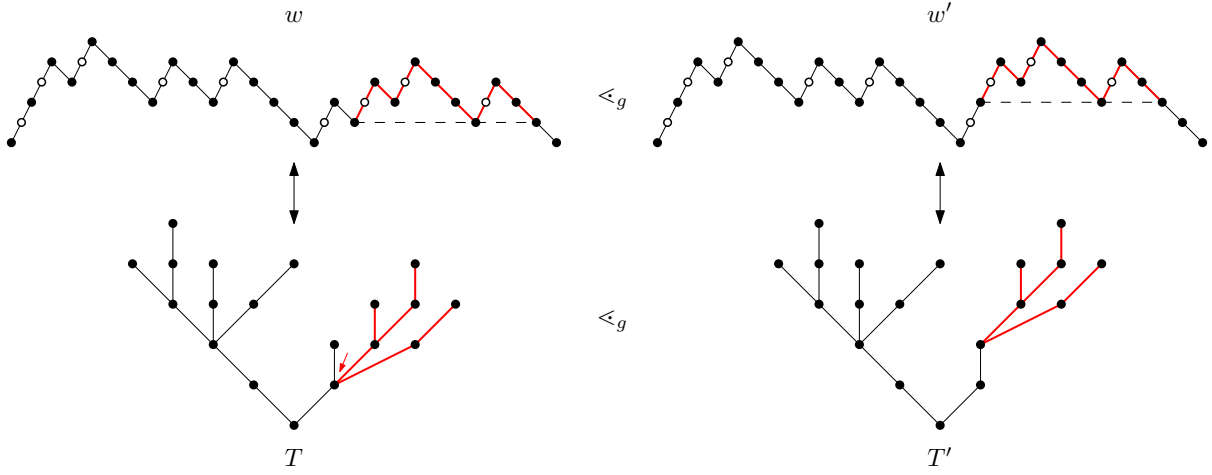

  \centering
  \insertfig{1}{1}
  \caption{Example of a covering relation in the greedy $2$-Tamari poset of order $9$, in both forms of $2$-ballot paths and $2$-branch trees}
  \label{fig:greedy-cover}
\end{figure}

\begin{defn}[See~\cite{greedy-tamari-interval}] \label{defn:m-greedy-poset-tree}
  The restriction of the relation $\leqgreedy$ to $\mtree_n$ defines a poset called the \tdef{greedy $m$-Tamari poset}, whose cover relation is the restriction of $\covgreedy$.
\end{defn}

See \Cref{fig:greedy-poset} for an illustration of both $(\mdyck_n, \leqgreedy)$ and $(\mtree_n, \leqgreedy)$ for $m = 2$ and $n = 3$.

\begin{figure}
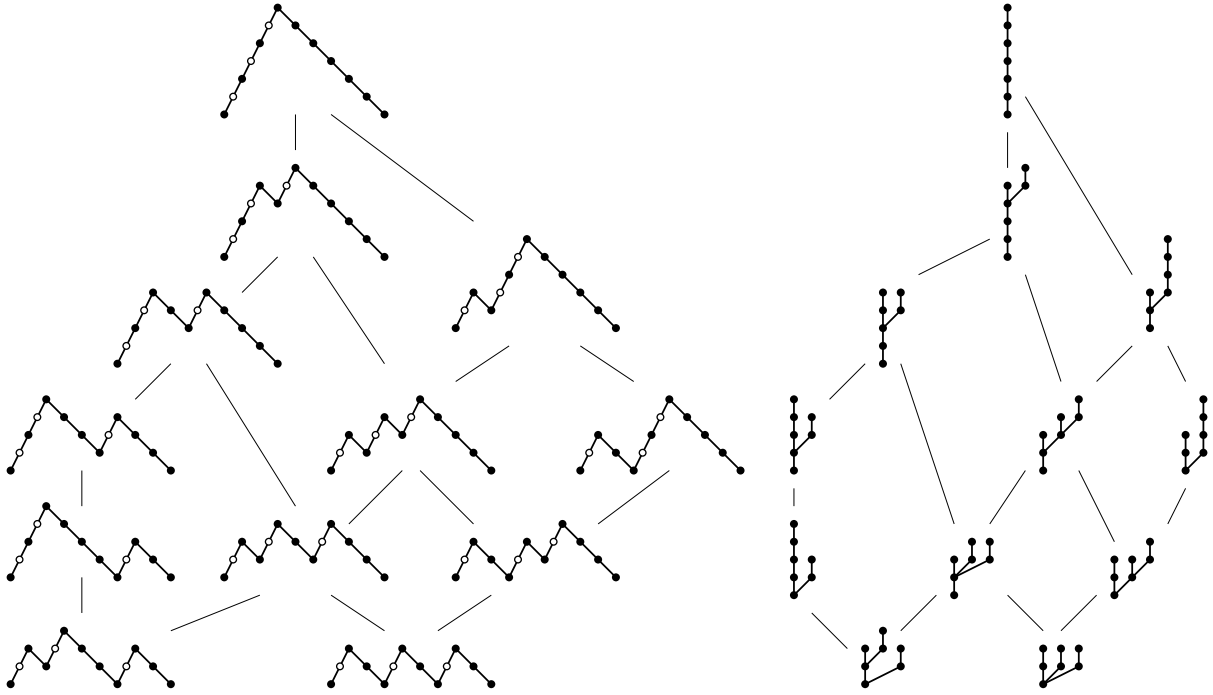

  \centering
  \insertfig{3}{1}
  \caption{Greedy $2$-Tamari poset of order $3$, with $2$-ballot paths as elements on the left, and with $m$-branch trees as elements on the right.}
  \label{fig:greedy-poset}
\end{figure}

We remark that the \tdef{trivial tree} with one vertex and no edge is an $m$-branch tree for all $m$. In fact, it is the only element of $\mtree_0$ for all $m$. Note that the maximal element in $(\mtree_n, \leqgreedy)$ is the linear tree with $mn$ edges. There are also $m$-versions of greedy Tamari intervals. A \tdef{greedy $m$-Tamari interval} of size $n$ is a pair $I = (T, T')$ with $T, T' \in \mtree_n$ such that $T \leqgreedy T'$. We denote by $\minterval_n$ the set of greedy $m$-Tamari intervals of size $n$. 

The generating function $F_{\minterval}(t, x) \equiv F_{\minterval}(t, x; p_1, p_2, \ldots)$ of greedy $m$-Tamari intervals is obtained by restricting of the sum in \Cref{eq:interval-weight} to $m$-branch trees while taking into account that the length of every ascent is a multiple of $m$. More precisely, comparing to \Cref{eq:interval-weight}, we define the weight $\intweight$ of a greedy $m$-Tamari interval $I$ by
\[
  \mintweight(I) = x^{\firstasc(I) / m} \prod_{i=2}^{r} p_{a_i / m},
\]
where both $\alf(I) = (a_1, \ldots, a_r)$ and $\firstasc(I)$ are taken for $I$ as a greedy Tamari interval in $(\tree_{mn}, \leqgreedy)$. The generating function is thus
\begin{equation}
  \label{eq:m-interval-gf}
  F_{\minterval}(t, x) = \sum_{n \geq 0} t^n \sum_{I \in \minterval_n} \mintweight(I).
\end{equation}

\subsection{Planar constellations and their decomposition}

A \tdef{planar $m$-constellation} is a special planar map where faces are colored black and white, with the outer face colored white and each edge bordering a black and a white face. Furthermore, black faces are called \tdef{hyperedge} and are all of degree $m$. White faces are called \tdef{hyperfaces} and are all of degree divisible by $m$. See \cref{fig:constellation} for an example of a planar constellation. The size of an $m$-constellation is the number of its hyperedges. Note that we allow the unique $m$-constellation of size $0$, which consists of a single vertex. We denote by $\mconst_n$ the set of $m$-constellations of size $n$. One may also encode elements in $\mconst_n$ with certain $(m+1)$-tuples of permutations in $S_n$, see \cite{lando-zvonkin}. One may identify planar $2$-constellations with rooted planar bipartite maps by transforming each hyperedge into an edge, as every hyperedge is of degree 2.

\begin{rem}
  The definition of $m$-constellations given here is in fact what is usually called \emph{$m$-hypermaps} in the literature. For $m$-constellations, there is a further condition on vertex coloring. However, in the planar case the two notions coincide, which allows us to slightly simplify the definition. Readers are again referred to \cite{lando-zvonkin} for further details.
\end{rem}

\begin{figure}
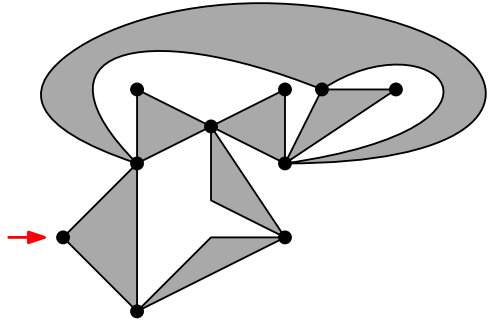

  \centering
  \insertfig{4}{0.4}
  \caption{Example of an $m$-constellation}
  \label{fig:constellation}
\end{figure}

We now define the generating function $F_{\mconst}$ of $m$-constellations that we use here. Given an $m$-constellations $C$, we denote by $\outdeg(C)$ the degree of its outer face divided by $m$, and we denote by $\deg(f)$ the degree of a given inner face $f$ divided by $m$. We then define $F_{\mconst}(t, x) \equiv F_{\mconst}(t, x; p_1, p_2, \ldots)$ by
\begin{equation}
  \label{eq:m-const-gf}
  F_{\mconst}(t, x) = \sum_{n \geq 0} t^n \sum_{C \in \mconst_n} x^{\outdeg(C)} \prod_{f\text{ inner face of }C} p_{\deg(f)}.
\end{equation}
In \cite[Section~4.1]{fang-thesis}, a recursive decomposition of planar $m$-constellations is given, leading to the following functional equation:
\begin{equation}
  \label{eq:m-const-decomp}
  F_{\mconst}(t, x) = 1 + tx(F_{\mconst} + \Omega)^m (1).
\end{equation}

\subsection{Refined and generalized conjecture on greedy $m$-Tamari intervals} 

Albeit much hope and effort, we did not succeed in generalizing our recursive decomposition of greedy Tamari intervals to the more general case of greedy $m$-Tamari intervals. However, supported by numerical evidence, we propose a more refined and generalized version of \cite[Conjecture~6.1]{greedy-tamari-interval}. We start with a generalization of greedy $m$-Tamari poset.

\begin{defn} \label{defn:m-r-tree}
  Given $m \geq 2$ and $ r\geq 0$, an \tdef{$(m, r)$-branch tree} is a rooted plane tree $T$ such that the length of the ascent of every leaf is divisible by $m$, except for the leftmost leaf, where it is of the form $ma + r$ for some integer $a \geq 0$. The number of edges in an $(m, r)$-branch tree $T$ is of the form $mn + r$ with some integer $n \geq 0$, since the ascent of leaves partition all edges in $T$, and we say that the \tdef{size} of $T$ is $n$. We denote by $\mrtree_n$ the set of $(m,r)$-branch trees of size $n$. We define the greedy $(m, r)$-Tamari poset to be $(\mrtree_n, \leqgreedy)$, with $\leqgreedy$ from \Cref{defn:m-greedy-poset-tree}. A \tdef{greedy $(m, r)$-Tamari interval} of size $n$ is a pair $(T, T')$ with $T, T' \in \mrtree_n$ such that $T \leqgreedy T'$, and we denote by $\mrinterval_n$ the set of all greedy $(m,r)$-Tamari intervals.
\end{defn}

We note that, again by the same argument for greedy $m$-Tamari posets, $\mrtree_n$ is also an order filter in $(\tree_{mn + r}, \leqgreedy)$, hence the cover relation of the greedy $(m, r)$-Tamari poset is still given by $\covgreedy$.

\begin{rem} \label{rem:m-r-tree}
  Like for the $m$-Tamari poset, our definition of greedy $(m,r)$-Tamari poset is also a specialization of the greedy $\nu$-Tamari poset defined in \cite{dermenjian}. More precisely, by taking contour walks, we see that $(\mrtree_n, \leqgreedy)$ is the $\nu$-Tamari poset with $\nu = NE^r(NE^m)^n$. By \Cref{defn:m-greedy-poset-tree} and \Cref{defn:m-r-tree}, we know that the greedy $(m, 0)$-Tamari poset of order $n$ is simply the greedy $m$-Tamari poset of order $n$, and the greedy $(m, m)$-Tamari poset of order $n$ is the greedy $m$-Tamari poset of order $n + 1$. We note that a greedy $(m, r)$-Tamari interval of size $n$ is also a greedy Tamari interval of size $mn + r$, as the order relation $\leqgreedy$ is the same.
\end{rem}

We now define the generating function $F_{\mrinterval{}}(t, x) \equiv F_{\mrinterval{}}(t, x; p_1, p_2, \ldots)$ for greedy $(m, r)$-Tamari intervals. Again, we need to slightly accommodate for the weights of these intervals. For $I \in \mrinterval_n$, noting that $\firstasc(I)$ takes the form $ma + r$ by \cref{defn:m-r-tree}, we define
\[
  \mrintweight(I) = x^{(\firstasc(I) - r) / m} \prod_{i = 2}^r p_{a_i / m},
\]
where both $\firstasc(I)$ and $\alf(I) = (a_1, \ldots, a_r)$ are taken for $I$ as a greedy Tamari interval in $(\tree_{mn + r}, \leqgreedy)$. We define
\begin{equation}
  \label{eq:mr-interval-gf}
  F_{\mrinterval{}}(t, x) = \sum_{n \geq 0} t^n \sum_{I \in \mrinterval_n} \mrintweight(I).
\end{equation}
We note that $F_{\mrinterval[0]{}}(t, x) = F_{\minterval}(t, x)$ and $F_{\mrinterval[m]{}}(t, x) = F_{\minterval}(t, x) - 1$. Numerical computation suggests the following conjecture.

\begin{conj} \label{conj:mr-interval}
  Given $m \geq 1$ and $1 \leq r \leq m-1$, the generating function $F_{\mrinterval{}}(t, x)$ satisfies
  \[
    F_{\mrinterval{}}(t, x) = (F_{\mrinterval[0]{}}(t, x) + \Omega)^{r+1}(1).
  \]
  For $r = 0$, we have
  \[
    F_{\mrinterval[0]{}}(t, x) = 1 + xt (F_{\mrinterval[0]{}}(t, x) + \Omega)^{m+1}(1).
  \]
  These equations can be recast into the quadratic recurrences
  \begin{equation}
    F_{\mrinterval[r+1]{}}(t, x) = (F_{\mrinterval[0]{}}(t, x) + \Omega)F_{\mrinterval{}}, \qquad 0 \leq r \leq m - 2, \label{eq:mr-conjrec1}
  \end{equation}
  \begin{equation}
    F_{\mrinterval[0]{}}(t, x) = 1 + xt(F_{\mrinterval[0]{}}(t, x) + \Omega)F_{\mrinterval[m-1]{}}. \label{eq:mr-conjrec2}
  \end{equation}
\end{conj}
The functional equation for $r = 0$ in \Cref{conj:mr-interval} is exactly the same as that of $(m+1)$-constellations, and it is already known to Bousquet-Mélou and Chapoton (private communication, 2024). However, the case of general $r$ seems to be new, which strongly suggests a recursive decomposition in steps of greedy $m$-Tamari intervals. It is hoped that a suitable generalization of the constructions in this paper might prove the recursions in the form of quadratic recurrences given by \cref{eq:mr-conjrec1,eq:mr-conjrec2}.

\section*{Acknowledgment} The authors thank Mireille Bousquet-Mélou for informative discussions.

\printbibliography

\end{document}